\documentclass[12pt,reqno]{amsart}

\usepackage{amstext}
\usepackage{amsthm}
\usepackage{amsmath}
\usepackage{amssymb}
\usepackage{latexsym}
\usepackage{amsfonts}
\usepackage{graphicx}
\usepackage{texdraw}
\usepackage{graphpap}
\usepackage{enumerate}

\usepackage[pagebackref,hypertexnames=false, colorlinks, citecolor=red, linkcolor=red]{hyperref} 
\usepackage[backrefs]{amsrefs}

\input txdtools

\begin{document}

\newcommand{\ci}[1]{_{ {}_{\scriptstyle #1}}}

\newcommand{\norm}[1]{\ensuremath{\left\|#1\right\|}}
\newcommand{\abs}[1]{\ensuremath{\left\vert#1\right\vert}}
\newcommand{\ip}[2]{\ensuremath{\left\langle#1,#2\right\rangle}}
\newcommand{\p}{\ensuremath{\partial}}
\newcommand{\pr}{\mathcal{P}}

\newcommand{\pbar}{\ensuremath{\bar{\partial}}}
\newcommand{\db}{\overline\partial}
\newcommand{\D}{\mathbb{D}}
\newcommand{\B}{\mathbb{B}}
\newcommand{\Sp}{\mathbb{S}}
\newcommand{\T}{\mathbb{T}}
\newcommand{\R}{\mathbb{R}}
\newcommand{\Z}{\mathbb{Z}}
\newcommand{\C}{\mathbb{C}}
\newcommand{\N}{\mathbb{N}}
\newcommand{\scrH}{\mathcal{H}}
\newcommand{\scrL}{\mathcal{L}}
\newcommand{\td}{\widetilde\Delta}

\newcommand{\BB}{\mathcal{B}}
\newcommand{\HH}{\mathcal{H}}
\newcommand{\KK}{\mathcal{K}}
\newcommand{\LL}{\mathcal{L}}
\newcommand{\MM}{\mathcal{M}}
\newcommand{\FF}{\mathcal{F}}

\newcommand{\Om}{\Omega}
\newcommand{\La}{\Lambda}

\newcommand{\rk}{\operatorname{rk}}
\newcommand{\card}{\operatorname{card}}
\newcommand{\ran}{\operatorname{Ran}}
\newcommand{\osc}{\operatorname{OSC}}
\newcommand{\im}{\operatorname{Im}}
\newcommand{\re}{\operatorname{Re}}
\newcommand{\tr}{\operatorname{tr}}
\newcommand{\vf}{\varphi}
\newcommand{\f}[2]{\ensuremath{\frac{#1}{#2}}}

\newcommand{\kzp}{k_z^{(p,\alpha)}}
\newcommand{\klp}{k_{\lambda_i}^{(p,\alpha)}}
\newcommand{\TTp}{\mathcal{T}_p}

\newcommand{\vp}{\varphi}
\newcommand{\al}{\alpha}
\newcommand{\be}{\beta}
\newcommand{\la}{\lambda}
\newcommand{\li}{\lambda_i}
\newcommand{\lb}{\lambda_{\beta}}
\newcommand{\Bo}{\mathcal{B}(\Omega)}
\newcommand{\Bbp}{\mathcal{B}_{\beta}^{p}}
\newcommand{\Bbt}{\mathcal{B}(\Omega)}
\newcommand{\Lbt}{L_{\beta}^{2}}
\newcommand{\Kz}{K_z}
\newcommand{\kz}{k_z}
\newcommand{\Kl}{K_{\lambda_i}}
\newcommand{\kl}{k_{\lambda_i}}
\newcommand{\Kw}{K_w}
\newcommand{\kw}{k_w}
\newcommand{\Kbz}{K_z}
\newcommand{\Kbl}{K_{\lambda_i}}
\newcommand{\kbz}{k_z}
\newcommand{\kbl}{k_{\lambda_i}}
\newcommand{\Kbw}{K_w}
\newcommand{\kbw}{k_w}
\newcommand{\BL}{\mathcal{L}\left(\mathcal{B}(\Omega), L^2(\Om;d\sigma)\right)}

\newcommand{\af}{\mathfrak{a}}
\newcommand{\bb}{\mathfrak{b}}
\newcommand{\cc}{\mathfrak{c}}

\newcommand{\entrylabel}[1]{\mbox{#1}\hfill}

\newenvironment{entry}
{\begin{list}{X}%
  {\renewcommand{\makelabel}{\entrylabel}%
      \setlength{\labelwidth}{55pt}%
      \setlength{\leftmargin}{\labelwidth}
      \addtolength{\leftmargin}{\labelsep}%
   }%
}%
{\end{list}}


\numberwithin{equation}{section}

\newtheorem{thm}{Theorem}[section]
\newtheorem{lm}[thm]{Lemma}
\newtheorem{cor}[thm]{Corollary}
\newtheorem{conj}[thm]{Conjecture}
\newtheorem{prob}[thm]{Problem}
\newtheorem{prop}[thm]{Proposition}
\newtheorem*{prop*}{Proposition}

\theoremstyle{remark}
\newtheorem{rem}[thm]{Remark}
\newtheorem*{rem*}{Remark}
\newtheorem{example}[thm]{Example}

\title{A Reproducing Kernel Thesis for Operators on Bergman-type Function Spaces}

\author[M. Mitkovski]{Mishko Mitkovski$^\dagger$}
\address{Mishko Mitkovski, Department of Mathematical Sciences\\ Clemson University\\ O-110 Martin Hall, Box 340975\\ Clemson, SC USA 29634}
\email{mmitkov@clemson.edu}
\urladdr{http://people.clemson.edu/~mmitkov/}
\thanks{$\dagger$ Research supported in part by National Science Foundation DMS grant \# 1101251.}

\author[B. D. Wick]{Brett D. Wick$^\ddagger$}
\address{Brett D. Wick, School of Mathematics\\ Georgia Institute of Technology\\ 686 Cherry Street\\ Atlanta, GA USA 30332-0160}
\email{wick@math.gatech.edu}
\urladdr{www.math.gatech.edu/~wick}
\thanks{$\ddagger$ Research supported in part by National Science Foundation DMS grant \# 0955432.}


\begin{abstract}
In this paper we consider the reproducing kernel thesis for boundedness and compactness for various operators on Bergman-type spaces. In particular, the results in this paper apply to the weighted Bergman space on the unit ball, and more generally to weighted Fock spaces.  
\end{abstract}

\maketitle

\section{Introduction}

The goal of this paper is to show that in a wide variety of classical function spaces various properties of a given operator can be determined by examining only its behavior on the normalized reproducing kernels of that space. In other words, the results in this paper may be viewed as ``Reproducing Kernel Thesis'' (RKT) statements. 
We develop a unified approach to tackle problems of this kind which works in a variety of classical function spaces.  We introduce a special class of Reproducing Kernel Hilbert Spaces (RKHS) as an abstract framework for our results. We call them Bergman-type spaces due to their similarities to the classical Bergman space. Two prime examples are the Bergman space and the Bargmann-Fock space. However, our definition is much more general and includes the weighted versions of these spaces, Bergman spaces on more complicated domains etc.

We consider two types of RKT statements; RKT for boundedness and RKT for compactness. We show that in our setting RKT for boundedness is almost true in the sense that a condition slightly stronger than the optimal one is enough to conclude boundedness. The first result of this kind seems to be due to Bonsall~\cite{Bo} who showed that Hankel operators (with BMO symbols) on the classical Hardy space satisfy RKT for boundedness. This result, now known as the ``Reproducing Kernel Thesis'', was later extended to other function spaces~\cites{Ax, Bar, Sm, Zor}, with a similar or a slightly stronger assumption on the operator. In many cases, the strongest form of the RKT for boundedness proved to be wrong~\cite{Naz} for general operators, and some sort of strengthening of the original assumption is usually necessary. Here we provide a reasonable condition for general operators on Bergman-type spaces under which RKT for boundedness holds. Furthermore, we show that under the same condition operators on Bergman-type spaces satisfy the RKT for compactness. Namely, every operator which sends a weakly convergent sequence of normalized reproducing kernels into a strongly convergent one must be compact. Results of this kind are numerous in the literature. Here we mention only a few and refer the reader to \cites{AZ2, YS, E92, E, I, LH, R, St, MZ, CLNZ} and the references in those papers for more examples of these results. The RKT for compactness was first proved to hold for  every Toeplitz operator on the classical Bergman space of the disc by Zheng \cite{ZhD}. It was later proved by Stroethoff and Zheng that the same is true for Hankel operators on the Bergman space, see \cite{SZ}. These two results were generalized by Axler and Zheng in \cite{AZ}. They proved that every finite sum of finite products of Toeplitz operators also satisfy the RKTC. In \cite{Sua} Su\'arez, in certain sense, closed the problem in the classical Bergman setting by showing that the Toeplitz algebra is exactly the class for which the RKT for compactness holds. Su\'arez's results were later extended to various different function spaces and settings. Namely, the same results were shown to be true for the Bargmann-Fock space \cite{BI}, Bergman spaces on the disc and unit ball with classical weights \cite{MSW}, and the weighted Bergman spaces on the polydisc \cite{MW}. Here, by introducing the notion of the Bergman-type space we are able to give a unified approach to many of these results, very often with shorter and more informative proofs. However, our results also hold for Bergman spaces on certain bounded symmetric domains for which the RKT was previously not known. Finally, we want to mention that a different general approach which was applied to similar questions about Hankel operators was offered in~\cite{JPR}.

The paper is organized as follows. In the next section we give a precise definition of Bergman-type spaces and prove some of their basic properties. In the following section we treat the RKT for boundedness. In particular, we show how our general result can be used to deduce a variety of classical results for boundedness of Toeplitz and Hankel operators. In the last section we treat the RKT for compactness. First, we prove the crucial localization property building on some of the ideas initiated by Axler and Zheng and improved later by Su\'arez. We then use this localization property to estimate the essential norm of a large class of operators and, in particular, to prove the RKT for compactness in our general setting. Finally, we show that for a certain subclass of operators the main assumption in the RKT for compactness can be further relaxed. It is quite possible that this subclass generates the whole Toeplitz algebra, but we don't address this question here.

\section{Bergman-type spaces}

We introduce a large class of reproducing kernel Hilbert spaces that will form an abstract framework for our results. Due to their similarities with the classical Bergman space we call them Bergman-type spaces.  In defining the key properties of these spaces, we use the standard notation that $A\lesssim B$ to denote that there exists a constant $C$ such that $A\leq C B$.  And, $A\simeq B$ which means that $A\lesssim B$ and $B\lesssim A$.

Below we list the defining properties of these spaces.

\begin{itemize}

\item[\label{A1} A.1] Let $\Om$ be a domain (connected open set) in $\C^n$ which contains the origin. We assume that for each $z\in\Om$, there exists an involution $\varphi_z \in \textnormal{Aut}(\Om)$ satisfying $\varphi_z(0)=z$. 

\item[\label{A2} A.2] We assume the existence of a metric $d$ on $\Omega$ which is quasi-invariant under $\vp_z$, i.e., $d(u,v)\simeq d(\vp_z(u),\vp_z(v))$ with the implied constants independent of $u,v\in\Om$. In addition, we assume that the metric space $(\Om, d)$ is separable and finitely compact, i.e., every closed ball in $(\Om, d)$ is compact.   As usual, we denote by $D(z, r)$ the disc centered at $z$ with radius $r$ with respect to the metric $d$. 

\item[\label{A3} A.3] We assume the existence of a finite Borel measure $\sigma$ on $\Om$ and define $\Bo$ to be the space of holomorphic functions on $\Om$ equipped with the $L^2(\Om;d\sigma)$ norm. We assume that $\Bo$ is a RKHS and denote by $\Kz$ and $\kz$ the reproducing and the normalized reproducing kernel in $\Bo$. Everywhere in the paper $\norm{\,\cdot\,}$ and $\ip{\,\cdot}{\cdot\,}$ will always denote the norm and the inner product in $L^2(\Om;d\sigma)$. We will assume that $\norm{K_z}$  is continuous as a function of $z$ taking $(\Om, d)$ into $\R$.
\end{itemize}

We will say that $\Bo$ is a \textit{Bergman-type space} if in addition to \hyperref[A1]{A.1}-\hyperref[A3]{A.3} it also satisfies the following properties.

\begin{itemize}
\item[\label{A4} A.4] We assume that the measure $d\la(z):=\norm{K_z}^2d\sigma(z)$ is quasi-invariant under all $\vp_z$, i.e., for every Borel set $E\subset\Om$  we have $\la(E)\simeq \la(\vp_z(E))$ with the implied constants independent of $z\in\Om$. 

\item[\label{A5} A.5] We assume that $$\abs{\ip{k_z}{k_w}}\simeq \frac{1}{\norm{K_{\vp_z(w)}}},$$ with the implied constants independent of $z, w\in\Om$. 

\item[\label{A6} A.6] We assume that there exists a positive constant  $\kappa<2$ such that \begin{equation} \label{propA6}
\int_{\Om} {\frac{\abs{\ip{K_z}{K_w}}^{\frac{r+s}{2}}}{\norm{K_z}^s\norm{K_w}^r}\,d\la(w)}\leq C = C(r,s) < \infty, \; \; \forall z \in \Om
\end{equation}
for all $r>\kappa>s>0$ or that~\eqref{propA6} holds for all $r=s>0$.   In the latter case we will say that $\kappa =0$. These will be called the Rudin-Forelli estimates for $\Bo$.

\item[\label{A7} A.7] We assume that $\lim_{d(z,0)\to\infty}\norm{K_z}=\infty$.

\item[\label{A8} A.8] We assume that the metric space $(\Om, d)$ has finite asymptotic dimension in the sense of Gromov~\cite{Gro}. More precisely, we assume that there exists an integer $N>0$ such that for any $r>0$ there is a covering $\FF_r=\{F_j\}$ of $\Om$ by disjoint Borel sets satisfying
\begin{enumerate}
\item[\label{Finite} \textnormal{(1)}] every point of $\Om$ belongs to at most $N$ of the sets $G_j:=\{z\in\Om: d(z, F_j)\leq r\}$,
\item[\label{Diameter} \textnormal{(2)}] $\sup_j \textnormal{diam}_d\, F_j <\infty$.
\end{enumerate}

\end{itemize}

We say that $\Bo$ is a \textit{strong Bergman-type space} if we have $=$ instead of $\simeq$ everywhere in \hyperref[A1]{A.1}-\hyperref[A5]{A.5}. 

\subsection{Some Examples}
Below we list some examples of Bergman-type spaces which satisfy conditions \hyperref[A1]{A.1}-\hyperref[A7]{A.7}. All of them actually represent strong Bergman-type spaces.

\begin{example} The classical Bergman space on the disc. In this case $\Om=\D$, $\vp_z(w)=\frac{z-w}{1-\bar{z}w}$, and $d\sigma$ is the normalized Lebesgue measure on $\D$ . The metric $d$ is the usual hyperbolic metric on the disc. All the conditions \hyperref[A1]{A.1}-\hyperref[A8]{A.8} are well-known to true. Notice that $\kappa=1$. Condition \hyperref[A5]{A.5} in this case is the well-known ``magic''  identity:
$$
1-\abs{\varphi_z(w)}^2=\frac{(1-\abs{z}^2)(1-\abs{w}^2)}{\abs{1-\overline{z}w}^2}.
$$
\end{example}

\begin{example} The weighted Bergman space $A_{\al}(\B_n)$, $\al>-1$ over the ball $\B_n$ with the classical weight $(1-\abs{z}^2)^{\al}$. In this case $\Om=\B_n$, $\vp_z(w)$ is the automorphism of the ball such that $\vp_z(0)=z$, and $d\sigma$ is the normalized Lebesgue measure on $\B_n$ weighted by $(1-\abs{z}^2)^{\al}$. The metric $d$ is the usual hyperbolic metric on the ball. Again, all the conditions \hyperref[A1]{A.1}-\hyperref[A8]{A.8} are well-known to be true. In this case it is not hard to check that $\kappa=\frac{2n}{n+1+\al}$.
\end{example}


\begin{example}\label{example} The Bergman space over any bounded symmetric domain which satisfies \hyperref[A6]{A.6} and \hyperref[A8]{A.8}. It is clear that \hyperref[A1]{A.1} is satisfied with the obvious choice of $\vp_z$. The metric $d$ on $\Om$ is taken to be the classical Bergman metric. With this choice of $d$,  \hyperref[A2]{A.2} is also satisfied. The measure $\sigma$ is just  taken to be the usual volume (Lebesgue) measure. In this case $\lambda$ is just the invariant measure on $\Om$. Therefore both  \hyperref[A3]{A.3} and  \hyperref[A4]{A.4} hold. It is well known also that  \hyperref[A5]{A.5} and  \hyperref[A7]{A.7} are true (see e. g. ~\cite{BCZ}). The Rudin-Forelli estimates \hyperref[A6]{A.6} are known to hold for a wide variety of bounded symmetric domains as shown in ~\cites{FK, CR}. This example of course contains all the previous examples as a special case.

\end{example}

\begin{example} The classical Fock space. In this case $\Omega=\C^n$,  $\vp_z(w)=z-w$ $d\sigma(z)=\frac{\alpha}{\pi}e^{-\alpha\abs{z}^2}dv(z)$, where $dv(z)$ is the usual Lebesgue measure on $\C^n$. The metric $d$ is just the Euclidian metric. All the properties \hyperref[A1]{A.1}-\hyperref[A8]{A.8} are also well-known to be true. Notice also that for this space  $\kappa=0$.
\end{example}

\subsection{Projection Operators on Bergman-type Spaces} 
It is easy to see that the orthogonal projection of $L^2(\Om;d\sigma)$ onto $\Bo$ is given by the integral operator
$$
P(f)(z):=\int_{\Om}\ip{\Kbw}{\Kbz}f(w)d\sigma(w).
$$
Therefore, for all $f\in\Bo$ we have $f(z)=\int_{\Om}\ip{f}{\kbw}\kbw(z)\,d\lambda(w).$
Moreover, $$\norm{f}^2=\int_{\Om}\abs{\ip{f}{\kw}}^2\,d\la(w).$$

We next show that if $\kappa>0$, $P$ is bounded as an operator on $L^p(\Om;d\sigma)$ for $1<p<\infty$ and if $\kappa=0$, $P$ is bounded as an operator on $L^p(\Om;\frac{d\sigma(w)}{\norm{K_w}^p})$. To do this we will need the following result which will be useful later as well. 

\begin{lm}\label{RF} For all $r, s\in\R$ the following quasi-identity holds 
\begin{equation}
\int_{\Om} {\frac{\abs{\ip{K_z}{K_w}}^{\frac{r-s}{2}}}{\norm{K_w}^r}\,d\la(w)}\simeq\int_{\Om} {\frac{\abs{\ip{K_z}{K_w}}^{\frac{r+s}{2}}}{\norm{K_z}^s\norm{K_w}^r}\,d\la(w)}
\end{equation}
where the implied constants are independent of $z\in\Om$ and may depend on $r,s$.
\end{lm}

\begin{proof} The proof is a change of variables and then using \hyperref[A5]{A.5} two times.  Indeed,
\begin{eqnarray*} \int_{\Om}\frac{ \abs{ \ip{K_z}{K_w}}^{\frac{r+s}{2}}}{\norm{K_z}^s\norm{K_w}^r}d\la(w) &\simeq& 
  \int_{\Om}\frac{ \abs{ \ip{K_z}{K_{\vp_z(u)}}}^{\frac{r+s}{2}}}{\norm{K_z}^s\norm{K_{\vp_z(u)}}^r}\,d\la(\vp_z(u))\\ &\simeq& \int_{\Om}\frac{ \norm{K_z}^{\frac{r+s}{2}}}{\norm{K_z}^s\norm{K_{\vp_z(u)}}^r\abs{k_z(u)}^{\frac{r+s}{2}}}\,d\la(u)\\
 &\simeq&  \int_{\Om}\frac{ \norm{K_z}^{\frac{r-s}{2}}\abs{\ip{k_z}{k_u}}^r}{\norm{K_{u}}^{\frac{r+s}{2}}\abs{\ip{k_z}{k_u}}^{\frac{r+s}{2}}}\,d\la(u)\\
 &=& \int_{\Om} {\frac{\abs{\ip{K_z}{K_u}}^{\frac{r-s}{2}}}{\norm{K_u}^r}\,d\la(u)}.
\end{eqnarray*}
\end{proof}

We will also use the following classical result.  See for example \cite{Zhu} for a proof.

\begin{lm}[Schur's Test]
\label{Schur}
Let $(X,\mu)$ and $(X,\nu)$ be measure spaces, $R(x,y)$ a non-negative measurable function on $X\times X$, $1<p<\infty$ and $\frac{1}{p}+\frac{1}{q}=1$.  If $h$ is a positive function on $X$ that is measurable with respect to $\mu$ and $\nu$ and $C_p$ and $C_q$ are positive constants such that
\begin{itemize}
\item[] $$\int_{X} R(x,y) h(y)^q\,d\nu(y)\leq C_q h(x)^q\textnormal{ for } \mu\textnormal{-almost every } x;$$
\item[] $$\int_{X} R(x,y) h(x)^p\,d\mu(x)\leq C_p h(y)^p\textnormal{ for } \nu\textnormal{-almost every } y,$$
\end{itemize}
then $Tf(x)=\int_{X} R(x,y) f(y)\,d\nu(y)$ defines a bounded operator $T:L^p(X;\nu)\to L^p(X;\mu)$ with $\norm{T}_{L^p(\nu)\to L^p(\mu)}\leq C_q^{\frac{1}{q}}C_p^{\frac{1}{p}}$.
\end{lm}

\begin{lm}\label{proj} Let $P(f)(z):=\int_{\Om}\ip{\Kbw}{\Kbz}f(w)d\sigma(w)$ be the projection operator.
\begin{enumerate}[\textnormal{(}a\textnormal{)}]
\item If $\kappa=0$ then $P$ is bounded as an operator from $L^p(\Om;\frac{d\lambda(w)}{\norm{K_w}^p})$ into $L^p(\Om;\frac{d\lambda(w)}{\norm{K_w}^p})$ for all $1\leq p\leq \infty$. 
\item If $\kappa>0$ then $P$ is bounded as an operator from $L^p(\Om;d\sigma)$ into $L^p(\Om;d\sigma)$ for all $1<p<\infty$
\end{enumerate}
\end{lm}
\begin{proof} (a) If $p=1$ then 
\begin{eqnarray*} \int_{\Om}\abs{Pf(z)}\frac{d\lambda(z)}{\norm{K_z}}&=&\int_{\Om}\abs{\int_{\Om}\ip{\Kbw}{\Kbz}f(w)d\sigma(w)}\frac{d\lambda(z)}{\norm{K_z}}\\
&\leq& \int_{\Om}\int_{\Om}\abs{\ip{\Kbw}{\Kbz}}\abs{f(w)}d\sigma(w)\frac{d\lambda(z)}{\norm{K_z}}\\
&=& \int_{\Om}\int_{\Om}\abs{\ip{\Kbw}{\Kbz}}\frac{d\lambda(z)}{\norm{K_z}}\abs{f(w)}d\sigma(w)\\
&\leq& C(1,1)\int_{\Om}\norm{K_w}\abs{f(w)}d\sigma(w)\\
&=& C(1,1)\int_{\Om}\abs{f(w)}\frac{d\lambda(w)}{\norm{K_w}},
\end{eqnarray*} where $C(1,1)$ is the constant from \hyperref[A6]{A.6}. So $P$ is bounded on $L^1(\Om;\frac{d\lambda(w)}{\norm{K_w}^1})$.

If $p=\infty$ then 
\begin{eqnarray*}\abs{Pf(z)}&=&\sup_{z\in\Om}\abs{\int_{\Om}\ip{\Kbw}{\Kbz}f(w)d\sigma(w)}\\
&\leq& \int_{\Om}\abs{\ip{\Kbw}{\Kbz}}\abs{f(w)}\frac{d\lambda(w)}{\norm{K_w}^2}\\
&\leq& \sup_{w\in\Om}\frac{\abs{f(w)}}{\norm{K_w}}\int_{\Om}\abs{\ip{\Kbw}{\Kbz}}\frac{d\lambda(w)}{\norm{K_w}}\\
&\leq& C(1,1)\norm{K_z} \sup_{w\in\Om}\frac{\abs{f(w)}}{\norm{K_w}}.
\end{eqnarray*} By interpolation we obtain that $P$ is bounded on $L^p(\Om;\frac{d\lambda(w)}{\norm{K_w}^p})$ for all $1\leq p\leq \infty$. 

(b) For a given $p$ choose $\gamma>0$ so that $p\gamma<\min(\kappa, 2-\kappa)$ and $q\gamma=\frac{p\gamma}{p-1}<\min(\kappa, 2-\kappa)$, where $\kappa<2$ is the constant appearing in \hyperref[A6]{A.6}. Notice that since $\kappa<2$, we have that such a number $\gamma$ exists. 

We will use the Schur's Test (Lemma~\ref{Schur}) with
\[R(z,w) = \abs{ \ip{\Kbz}{\Kbw} }, \hspace{0.2cm} h(z) \equiv \norm{\Kbz}^{\gamma}, \hspace{0.2cm}  X = \Om, \hspace{0.2cm} d\mu(w) = d\nu(w) = d\sigma(w),\]
where $\al>0$ is to be specified later.  Then we have

 \begin{eqnarray*} \int_{\Om} \abs{ \ip{\Kbz}{\Kbw}}\norm{\Kbw}^{\al}\frac{d\la(w)}{\norm{\Kbw}^2}
&=& \norm{\Kbz}\int_{\Om} \abs{\ip{\kbz}{k_{\vf_z(u)}}}\norm{K_{\vf_z(u)}}^{\al-1}\,d\la(\vf_z(u)) \\
&\simeq& \norm{\Kbz}\int_{\Om}\frac{1}{\norm{K_{\vp_z(\vp_z(u))}}}\abs{\ip{\kbz}{k_{u}}}^{1-\al} \,d\la(u) \\
&=& \norm{\Kbz}\int_{\Om}\frac{1}{\norm{K_u}}\abs{\ip{\kbz}{k_{u}}}^{1-\al}d\la(u) \\
&=& \norm{\Kbz}^{\al} \int_{\Om}\frac{\abs{\ip{\Kbz}{K_{u}}}^{1-\al}}{\norm{K_u}^{2-\al}}\,d\la(u). 
\end{eqnarray*}
If $0<\al<\min(\kappa, 2-\kappa)$ we can use Lemma~\ref{RF}, and \hyperref[A6]{A.6} with $r=2-\al$ and $s=\al$ to obtain that the integral above is bounded independent of $z$. Therefore, by taking $\al=p\gamma$ and $\al=q\gamma$ we conclude that both of the conditions in the Schur's Test are satisfied and consequently our operator $P$ is bounded on $L^p(\Om;d\sigma)$. 

 \end{proof}

\subsection{Translation Operators on Bergman-type Spaces}
 
For each $z\in\Om$ we define an adapted translation operator $U_z$ on  $\Bo$ by 
$$ U_zf(w):=f(\vp_z(w))k_z(w).$$
Using \hyperref[A1]{A.1}-\hyperref[A5]{A.5} it is easy to check that $\norm{U_zf}\simeq \norm{f}$ with implied constants independent of $z$. In addition, each $U_z$ is invertible with the inverse given by 
$$ U_z^{-1}f(w):=\frac{f(\vp_z(w))}{k_z(\vf_z(w))}.$$
The inverse also satisfies $\norm{U^{-1}_zf}\simeq \norm{f}$. Therefore, for every $f\in\Bo$ we have 
$$ \norm{f}^2=\ip{U_z^*f}{U_z^{-1}f}\leq \norm{U^*_zf}\norm{U_z^{-1}f}\lesssim \norm{U_z^*f}\norm{f}.$$
This implies that also $\norm{U^*_zf}\simeq \norm{f}$.

\begin{lm}\label{identities} The following quasi-equalities hold: 
\begin{itemize}
\item[(a)] $\norm{U_zf}\simeq \norm{f}$,
\item[(b)] $\abs{U_z^2f}\simeq\abs{f}$, 
\item[(c)] $|U_z^*k_w|\simeq|k_{\vp_z(w)}|$. 
\end{itemize}
\end{lm}

\begin{proof} (a) \begin{eqnarray*} \norm{U_zf}^2 &=& \int_{\Om}\abs{f(\vp_z(w))}^2\abs{k_z(w)}^2d\sigma(w) = \int_{\Om}\abs{f(\vp_z(w))}^2\abs{\ip{k_z}{k_w}}^2d\la(w)\\ &\simeq& \int_{\Om} \frac{\abs{f(\vp_z(w))}^2}{\norm{K_{\vp_z(w)}}^2}d\la(w)\simeq\norm{f}^2
\end{eqnarray*}
(b) 
\begin{eqnarray*} \abs{U_z^2f(w)}&=&\abs{U_z((f\circ k_z)k_z)(w)}=\abs{f(\vp_z(\vp_z(w)))k_z(\vp_z(w))k_z(w)}\\&=&\abs{f(w)}\abs{\ip{k_z}{K_{\vp_z(w)}}}\abs{\ip{k_z}{K_w}}=\abs{f(w)}
\end{eqnarray*}
(c) $$\abs{U_z^*k_w(u)}=\abs{\ip{k_w}{U_zK_u}}=\abs{K_u(\vp_z(w))}\abs{\ip{k_z}{k_w}}\simeq\abs{\ip{K_u}{k_{\vp_z(w)}}}=\abs{k_{\vp_z(w)}(u)}.$$

\end{proof}

In case of a strong Bergman-type space $U_z$ are actually unitary operators. Moreover, in this case $U_z^2=I$ and  $\abs{U_z^*k_w}=\abs{k_{\vp_z(w)}}=\abs{U_z k_w}$.

For any given operator $T$ on $\Bo$ and $z\in\Om$ we define $T^z:=U_zTU^*_z$.

\subsection{Toeplitz Operators on Bergman-type Spaces}

Let $M_u$ denote the operator of multiplication by the function $u$, $M_u(f):=uf$. We define a Toeplitz operator with symbol $u$ to be the operator given by
$$
T_u:=PM_u,
$$ where $P$ is the usual projection operator onto $\Bo$. For a function $u:\Om\to \C$ such that $u(\vp_z(w))\in L^2(\Om;d\sigma(w))$ the product $u(w)k_z(w)$ belongs in $L^2(\Om;d\sigma)$. So, if this is true for all $z\in\Om$, the Toeplitz operator $T_u$ will at least be densely defined on $\Bo$ because it will be well defined on the span of the normalized reproducing kernels. In the next section we will provide a condition on $u$ which will guarantee that $T_u$ is bounded. 

In the classical case when $u$ is bounded it is immediate to see that $\norm{T_u}\leq \norm{u}_{L^\infty}$.  These Toeplitz operators are the key building blocks of an important object for this paper, the Toeplitz algebra $\mathcal{T}_{L^\infty}$ associated to the symbols in $L^\infty(\Om)$.  For the family of function $u_{j,l}\in L^\infty(\Om)$ with $1\leq j\leq J$ and $1\leq l\leq L$ one considers the (algebraic or aggregate) operators
$$
T=\sum_{l=1}^L \prod_{j=1}^J T_{u_{j,l}},
$$
which are nothing other then finite sums of finite products of Toeplitz operators.  Then $\mathcal{T}_{L^\infty}$ is the closure of these operators, 
$$
\mathcal{T}_{L^\infty}=\textnormal{clos}_{\mathcal{L}(\Bo,\Bo)}\left\{\sum_{l=1}^L \prod_{j=1}^J T_{u_{j,l}}: u_{j,l}\in L^\infty(\Om), J, L \textnormal{ finite}\right\}
$$
where the closure is in the operator norm on $\Bo$, i.e., $\norm{\,\cdot\,}_{\mathcal{L}(\Bo,\Bo)}$.

In the case of strong Bergman-type spaces conjugation by translations behaves particularly well with respect to Toeplitz operators. Namely, if $T=T_u$ is a Toeplitz operator then $T^z=T_{u\circ \varphi_{z}}$. Moreover, when  $T=T_{u_1}T_{u_2}\cdots T_{u_n}$ is a product of Toeplitz operators we have

$$ T^z=T_{u_1\circ \varphi_{z}}T_{u_2\circ \varphi_{z}}\cdots T_{u_n\circ \varphi_{z}}.$$

The following simple lemma will be used in what follows.

\begin{lm}\label{ToeplitzCompact}
For each bounded Borel set $G$ in $\Om$ the Toeplitz operator $T_{1_G}$ is compact on $\Bo$.
\end{lm}
\begin{proof} 
Since $G$ is bounded there exists a closed ball $B$ such that $G\subset B$. Let $\{f_n\}$ be a unit-norm sequence in $\Bo$. Since, by~\hyperref[A3]{A.3}, $\sup_{z\in B}\norm{K_z}<\infty$ we have that the sequence $\{f_n\}$ is bounded on $B$, i. e., there exists $M$ independent on $n$ such that $\sup_{z\in B}\abs{f_n(z)}\leq M$. Since the same is true for any for any compact set in place of $B$, by Montel's theorem, there exists a subsequence $\{f_{n_k}\}$ such that $\{M_{1_B}f_{n_k}\}$ converges uniformly. This implies that  $\{M_{1_B}f_{n_k}\}$ converges in $L^2(\sigma)$. Therefore, $\{M_{1_G}f_{n_k}\}$ converges in $L^2(\sigma)$, and consequently $\{T_{1_G}f_{n_k}\}$ converges in $\Bo$. 
\end{proof}

\subsection{Geometric Decomposition of \texorpdfstring{$(\Om, d, \la)$}{the Domain}}

The proof of the crucial localization result from Section~\ref{RKTComp} will make critical use of the fact that $(\Om, d)$ has finite asymptotic dimension in the sense of M. Gromov. The finiteness of the asymptotic dimension is a very useful property which is satisfied for a great variety of domains in $\C^n$. It was stated in~\cite{Gro} (and explicitly proved in~\cite{Roe}) that every Gromov hyperbolic geodesic metric space with bounded growth at some scale has finite asymptotic dimension. It is known that every bounded strictly pseudoconvex domain with $C^2$ smooth boundary equipped with the Bergman metric is Gromov hyperbolic~\cite{Bon}. It is also known that every bounded symmetric domain equipped with the Bergman metric has a bounded growth at some scale. In particular, this implies that the unit ball $\B_n$ equipped with the usual hyperbolic metric has bounded growth at some scale (see \cite{BonS}*{page 25}). Therefore, combining these results, we have that every bounded symmetric domain equipped with the Bergman metric has finite asymptotic dimension as stated in the introduction. Related results can be found in \cites{CR,Sua,MSW,BI} where it is shown that nice domains, such as the unit ball, polydisc, or $\C^n$ have this property. 

On the other hand, it is much easier to see that metric spaces that can be equipped with a doubling quasi-invariant measure also have finite asymptotic density. Below, we give a prove of this result.


First recall the following well-known fact valid for arbitrary metric spaces.
\begin{lm}[\cite{ARS}*{Lemma 7, page 18}]
\label{MetricDecomp}
Let $(X,d)$ be a separable metric space and $r>0$.  There is a denumerable set of points $\{x_j\}$ and a corresponding set of Borel subsets $\{Q_j\}$ of $X$ that satisfy
\begin{enumerate}
\item[\textnormal{(1)}] $X=\bigcup_{j} Q_j$;
\item[\textnormal{(2)}] $Q_j\cap Q_{j'}=\emptyset$ for $j\neq j'$;
\item[\textnormal{(3)}] $D(x_j, r)\subset Q_j\subset D(x_j, 2r)$.
\end{enumerate}
Here $D(x,r)$ denotes the open ball with center $x$ and radius $r>0$ in the metric space $(X,d)$.
\end{lm}

Using this fact we now prove that if $(\Omega, d)$ can be equipped with a quasi-invariant doubling measure $\lambda$, then it satisfies \hyperref[A8]{A.8}, i.e., it has finite asymptotic dimension in the sense of Gromov.
By \hyperref[A2]{A.2} the set $(\Omega, d)$ is a separable metric space, and so by Lemma \ref{MetricDecomp} for each $r>0$ there is a collection of points $\{x_j\}\in\Omega$ and Borel sets $F_j:=Q_j\subset\Omega$ so that $\mathcal{F}_r:=\{F_j\}$ is a disjoint covering of $\Omega$.  Also, by Lemma \ref{MetricDecomp} we have that
$$
\textnormal{diam}_d\, F_j \leq \textnormal{diam}_d\, D(x_j,2r) =4r,
$$
proving (\hyperref[Diameter]{2}) from \hyperref[A8]{A.8}.

It remains to prove (\hyperref[Finite]{1}) from \hyperref[A8]{A.8}, namely that we have the finite overlap property for the sets $G_j=\{z\in\Om: d(z, F_j)\leq r\}$.  Since we clearly have $G_j\subset D(x_j, 4r)$ it suffices to prove that the corresponding balls $D(x_j, 4r)$ have a finite intersection property.  Suppose that $z\in \cap_{l=1}^N D(x_l, 4r)$.  Since $(\Omega,d)$ is a metric space we have that
$$
\bigcup_{l=1}^N D(x_l, r)\subset D(x_k, 9r).
$$
Here we have fixed one of the balls to play a distinguished role, namely the ball centered at $x_k$.  By the construction we have that $D(x_l,r)$ are disjoint in $\Omega$.  By properties \hyperref[A1]{A.1} and \hyperref[A2]{A.2} we have that $\la\left(D(x_j, r)\right)\simeq \la\left(D(0, cr)\right)$ for any $r>0$, for all $j$, and for some constant $c>0$ independent of $j, r$.  So we have that
$$
\la\left(D(0, r)\right)N\simeq \sum_{l=1}^N \la\left(D(x_{l}, r)\right)\leq \la\left(D(x_{k}, 9r)\right)\simeq \la\left(D(0, 9cr)\right).
$$
This gives that
$$
N\lesssim \frac{\la\left(D(0, 9cr)\right)}{\la\left(D(0, r)\right)}
$$
and this is a constant depending only on the doubling constant of the measure $\la$.

\section{Reproducing kernel thesis for boundedness}

In this section we prove that many operators on Bergman-type spaces have a property reminiscent of the classical reproducing kernel thesis (RKT). Ideally we would like to show that the condition 
$\sup_{z\in\Om}\norm{Tk_z}<\infty$ is sufficient for $T$ to be bounded. However, this is most probably not true in general. Below, we show that if we impose a stronger condition we can get the desired boundedness. Our condition will be on $U_zTk_z$ instead on $Tk_z$.  We need to keep $U_z$ since we don't generally have that $\norm{U_z}_{L^p(\Om;d\sigma)}\simeq 1$ for $p>2$.

\begin{thm}\label{RKT} Let $T:\Bo\to\Bo$ be a linear operator defined a priori only on the linear span of the normalized reproducing kernels of $\Bo$. Assume that there exists an operator $T^*$ defined on the same span such that the duality relation $\ip{Tk_z}{k_w}=\ip{k_z}{T^*k_w}$ holds for all $z,w\in\Om$. Let  $\kappa$ be the constant from \hyperref[A6]{A.6}. If
\begin{equation} \sup_{z\in\Om}\norm{U_zTk_z}_{L^p(\Om;d\sigma)}<\infty \hspace{.3cm} \text{and} \hspace{.3cm}  \sup_{z\in\Om}\norm{U_zT^*k_z}_{L^p(\Om;d\sigma)}<\infty
\end{equation} for some $p>\frac{4-\kappa}{2-\kappa}$ then $T$ can be extended to a bounded operator on $\Bo$.  
\end{thm}

In the case of the unweighted Bargmann-Fock space this results corresponds to~\cite{CWZ}*{Theorem 4} and in the case of the unweighted Bergman space it corresponds to~\cite{DT}*{Proposition 2.5}.

\begin{rem} In view of Lemma~\ref{proj}, it would be interesting to see if for $\kappa=0$ the conditions 
\begin{equation} \sup_{z\in\Om}\norm{U_zTk_z}_{L^p(\Om;d\frac{\sigma(w)}{\norm{K_w}^p})}<\infty \hspace{.3cm} \text{and} \hspace{.3cm}  \sup_{z\in\Om}\norm{U_zT^*k_z}_{L^p(\Om;d\frac{\sigma(w)}{\norm{K_w}^p})}<\infty
\end{equation} for some $p>2$ are sufficient for $T$ to be bounded on $\Bo$. 
\end{rem}

\begin{proof}  
Since the linear space of all the normalized reproducing kernels is dense in $\Bo$ it will be enough to show that there exists a constant $C<\infty$ such that $\norm{Tf}\lesssim\norm{f}$ for all $f$ that are in the linear span of the reproducing kernels. Notice first that for any such $f$ we have 

\begin{eqnarray*} \norm{Tf}^2&=&\int_{\Omega}\abs{\ip{Tf}{K_z}}^2d\sigma(z)=\int_{\Omega}\abs{\ip{f}{T^*K_z}}^2d\sigma(z)\\
&\leq& \int_{\Omega}\abs{\int_{\Om}\ip{f}{K_w}\ip{K_w}{T^*K_z}d\sigma(w)}^2d\sigma(z)\\
&\leq& \int_{\Omega}\abs{\int_{\Om}\abs{\ip{K_w}{T^*K_z}}\abs{f(w)}d\sigma(w)}^2d\sigma(z).
\end{eqnarray*}
Consider the integral operator  $Rf(z):=\int_{\Om}  \abs{ \ip{T^*\Kbz}{\Kbw} }f(w)d\sigma(w)$. It is enough to show that this operator is bounded as an operator on $L^2(\Om;d\sigma)$. 
To do this we will use the Schur's Test (Lemma~\ref{Schur}) with
\[R(z,w) = \abs{ \ip{T^*\Kbz}{\Kbw} }, \hspace{0.15cm} h(z) \equiv \norm{\Kbz}^{\al/2}, \hspace{0.15cm}  X = \Om, \hspace{0.15cm} d\mu(z) =  d\nu(z) = d\sigma(z).\]
If $\kappa=0$ set $\al=\frac{4-2\kappa}{4-\kappa}=1$. If $\kappa>0$ choose $\alpha\in (\frac{2}{p}, \frac{4-2\kappa}{4-\kappa})$ such that $q(\alpha-\frac{2}{p})<\kappa$. The condition $p>\frac{4-\kappa}{2-\kappa}$ ensures that such $\alpha$ exists. 
Let $z\in\Om$ be arbitrary and fixed. We have

\begin{eqnarray*} \int_{\Om} R(z, w)\norm{\Kbw}^{\al} \,d\sigma(w) &=& \int_{\Om} \abs{ \ip{T^*\Kbz}{\Kbw}}\norm{\Kbw}^{\al} \,d\sigma(w)\\
&=&  \norm{\Kbz}\int_{\Om} \abs{\ip{T^*\kbz}{k_{\vf_z(u)}}}\norm{K_{\vf_z(u)}}^{\al-1}\,d\la(u) \\
&\simeq& \norm{\Kbz}\int_{\Om}\abs{\ip{T^*k_{z}}{U^*_zk_{u}}}\abs{\ip{\kbz}{k_{u}}}^{1-\al} \,d\la(u) \\
&=& \norm{\Kbz}\int_{\Om}\abs{\ip{U_zT^*k_z}{k_{u}}}\abs{\ip{\kbz}{k_{u}}}^{1-\al}\,d\la(u) \\
&=& \norm{\Kbz}^{\al} \int_{\Om}\frac{\abs{U_zT^*k_z(u)}\abs{\ip{\Kbz}{K_{u}}}^{1-\al}}{\norm{K_u}^{2-\al}}\,d\la(u). 
\end{eqnarray*}
Using Holder's inequality we obtain that the last expression is no greater than
$$
\norm{\Kbz}^{\al}\left(\int_{\Om}\abs{U_zT^*k_z(u)}^p\,d\sigma(u)\right)^{\frac{1}{p}}\left(\int_{\Om}\frac{\abs{\ip{\Kbz}{K_{u}}}^{q(1-\al)}}{\norm{K_u}^{q\left(2-\al-\frac{2}{p}\right)}}\,d\la(u)\right)^{\frac{1}{q}}.
$$
We can use Lemma~\ref{RF} to deduce that the second integral above is bounded by a constant independent of $z$. Indeed, set $r=q\left(2-\al-\frac{2}{p}\right)$ and $s=r-2q(1-\al)$. By the choice of $p$ and $\al$ we have $r=\frac{p\left(2-\al-\frac{2}{p}\right)}{p-1}=2-\frac{p\al}{p-1}>\kappa$ and $s=q\left(\al-\frac{2}{p}\right)<\kappa$ if $\kappa>0$, and $r=s>\kappa$ if $\kappa=0$. Therefore, $r$ and $s$ satisfy all the conditions of \hyperref[A6]{A.6}. So by Lemma~\ref{RF} we obtain:
$$
\int_{\Om} R(z, w)\norm{\Kbw}^{\al} \,d\sigma(w)\leq C \sup_{z\in\Om}\norm{U_zT^*k_z}_{L^p(\sigma)} \norm{\Kbz}^{\al},
$$ where $C$ is the constant from \hyperref[A6]{A.6}.  To check the second condition we use exactly the same estimates as above, but interchanging the roles of $T$ and $T^*$ and obtain that 

$$
\int_{\Om} R(z, w)\norm{K_z}^{\al} \,d\sigma(z)\leq C \sup_{z\in\Om}\norm{U_zTk_z}_{L^p(\sigma)} \norm{K_w}^{\al}.
$$
Therefore $R$ is bounded operator on $L^2(\Om;d\sigma)$ and hence $T$ is bounded too.
\end{proof}

\subsection{RKT for Toeplitz operators}
In the case when $T=T_u$ is a Toeplitz operator the conditions in the above Theorem  have a much nicer form. However, unless the symbol is holomorphic, we need to assume that the space $\Bo$ is a strong Bergman-type space.  Then we have the following corollary.

\begin{cor} Let $\Bo$ be a strong Bergman-type space.
If $\kappa>0$ and $T_u$ is a Toeplitz operator whose symbol $u$ satisfies $$\sup_{z\in\Om}\int_{\Om}\abs{u(\vp_z(w))}^pd\sigma(w)<\infty,$$ for some $p>\frac{4-\kappa}{2-\kappa}$ then $T_u$ is bounded on $\Bo$.
\end{cor}

\begin{proof} We will first show that $\abs{U_zT_uk_z(w)}=\abs{P(u\circ\vp_z)(w)}$. Notice that by~\ref{A5}, $\abs{k_0(w)}=1$ for all $w\in\Om$. Therefore, by holomorphicity it has to be constant on all $\Om$. Since $k_0(0)=\norm{K_0}>0$ we have that $k_0(w)\equiv 1$. By Lemma~\ref{proj} we then have 
$$  \sup_{z\in\Om}\norm{U_zTk_z}_{L^p(\Om;\sigma)}^p\lesssim  \sup_{z\in\Om}\int_{\Om}\abs{u(\vp_z(w))}^pd\sigma(w)<\infty. 
$$
Since $T^*_u=T_{\bar{u}}$ we will automatically have the other condition in Theorem~\ref{RKT} satisfied as well, and we will be done.
Next we show $\abs{U_zT_uk_z(w)}=\abs{P(u\circ\vp_z)(w)}$. The equality $\abs{U_z^*k_w}=\abs{k_{\vp_z(w)}}$ will be used several times.

\begin{eqnarray*}  \abs{U_zT_uk_z(w)} &=& \norm{K_w}\abs{\ip{U_zT_uk_z}{k_w}} \\
&=&\norm{K_w}\abs{\int_{\Om}T_uk_z(a)\overline{k_{\vp_z(w)}(a)}d\sigma(a)} \\
&=&\norm{K_w}\abs{\int_{\Om}u(a)k_z(a)\overline{k_{\vp_z(w)}(a)}d\sigma(a)} \\
&=&\norm{K_w}\abs{\int_{\Om}u(a)\ip{k_z}{k_a}\overline{\ip{k_{\vp_z(w)}}{k_a}}d\lambda(a)} \\
&=&\norm{K_w}\abs{\int_{\Om}u(\vp_z(b))\ip{k_z}{k_{\vp_z(b)}}\overline{\ip{k_{\vp_z(w)}}{k_{\vp_z(b)}}}d\lambda(b)}\\ 
&=&\abs{\int_{\Om}u(\vp_z(b))k_0(b)\overline{\ip{K_{w}}{K_{b}}}d\sigma(b)}=\abs{P(u\circ\vp_z)(w)}.
\end{eqnarray*}

\end{proof}

\subsection{RKT for product of Toeplitz operators with analytic symbols} We now derive a sufficient condition for boundedness for a product of Toeplitz operators  $T_fT_{\bar{g}}$ with $f, g\in \Bo$. In the case of the classical Bergman space a well known conjecture of D. Sarason claims that 
\begin{equation}
\sup_{z\in\Om}\int_{\D}\abs{f(\vp_z(w))}^{2}dA(w)\int_{\D}\abs{g(\vp_z(w))}^{2}dA(w)<\infty 
\end{equation}
 is sufficient for boundedness of $T_fT_{\bar{g}}$. It was proved by Zhang and Stroethoff in \cite{SZ2} that a slightly stronger condition is sufficient. They proved that if for some $\epsilon>0$ the following condition is satisfied 
\begin{equation}\label{Zhang}
\sup_{z\in\Om}\int_{\D}\abs{f(\vp_z(w))}^{2+\epsilon}dA(w)\int_{\D}\abs{g(\vp_z(w))}^{2+\epsilon}dA(w)<\infty,
\end{equation} then $T_fT_{\bar{g}}$ must be bounded. This was extended later to the weighted case setting in~\cite{PS}. Very recently Aleman, Pott, and Reguera~\cite{ARS} disproved the Sarason conjecture in the classical Bergman space. Here we provide another sufficient condition which is in general not comparable with~\eqref{Zhang}. The generality of our approach yields also a corresponding statement for the Bargmann-Fock space. It should be mentioned that in the Bargmann-Fock space a different, more explicit necessary and sufficient condition for boundedness is given in~\cite{CPZ}.   
\begin{cor} Let $\Bo$ be a strong Bergman-type space such that a product of any two reproducing kernels from $\Bo$ is still in $\Bo$. Let $f, g\in\Bo$.
If there exists $p>\frac{4-\kappa}{2-\kappa}$ such that  
\begin{equation*} 
\sup_{z\in\Om}\abs{g(z)}^p\int_{\Om}\abs{f(\vp_z(w))}^{p}d\sigma(w)<\infty,
\end{equation*} and 
\begin{equation*} 
\sup_{z\in\Om}\abs{f(z)}^p\int_{\Om}\abs{g(\vp_z(w))}^{p}d\sigma(w)<\infty,
\end{equation*}
then the operator $T_fT_{\bar{g}}$ is bounded on $\Bo$.
\end{cor}

\begin{proof} We only need to check that $T_fT_{\bar{g}}$ satisfies the conditions of Theorem~\ref{RKT}. 

We will prove first that if $g\in \Bo$, then $T_{\bar{g}}k_z=\overline{g(z)}k_z$.  Let $g$ be a finite linear combination of reproducing kernels. In this case, $gK_w\in\Bo$ for any reproducing kernel $K_w$. Therefore,
\begin{equation*} 
T_{\bar{g}}k_z(w) = \int_{\Om} \ip{K_u}{K_w}\overline{g(u)}k_z(u)d\sigma(u)=\overline{\ip{k_z}{K_wg}}=\frac{\overline{K_w(z)g(z)}}{\norm{K_z}}=\overline{g(z)}k_z(w).
\end{equation*}
Next, let $g\in\Bo$ be arbitrary. Fix $z, w\in\Om$. Let $\epsilon>0$. There exists $h$ which is a finite linear combination of reproducing kernels such that $\norm{g-h}<\epsilon$. We then have

\begin{eqnarray*} 
\abs{T_{\bar{g}}k_z(w)-\overline{h(z)}k_z(w)}&=&\abs{T_{\overline{g-h}}k_z(w)} = \abs{\int_{\Om} \ip{K_u}{K_w}(\overline{g(u)-h(u)})k_z(u)d\sigma(u)}\\
&\leq& \int_{\Om} \abs{K_w(u)k_z(u)}^2d\sigma(u)\norm{g-h}^2 <C(z, w)\epsilon^2.
\end{eqnarray*}
Moreover, $\abs{g(z)-h(z)}\leq \norm{K_z}\norm{g-h}<\norm{K_z}\epsilon$.
Since $z, w$ were fixed and $\epsilon>0$ was arbitrary we obtain $T_{\bar{g}}k_z=\overline{g(z)}k_z$.

So we have 
\begin{eqnarray*}  \abs{U_zT_fT_{\bar{g}}k_z(w)} &=& \norm{K_w} \abs{\ip{U_zT_fT_{\bar{g}}k_z}{k_w}} \\
&=&\norm{K_w}\abs{\overline{g(z)}\ip{U_zT_fk_z}{k_w}}\\
&=&\norm{K_w}\abs{\overline{g(z)}\ip{T_fk_z}{k_{\vp_z(w)}}} \\
&=&\norm{K_w}\abs{\overline{g(z)}\ip{k_z}{T_{\bar{f}}k_{\vp_z(w)}}} \\
&=&\norm{K_w}\abs{f(\vp_z(w))\overline{g(z)}\ip{k_z}{k_{\vp_z(w)}}}. 
\end{eqnarray*}

Above we used the identity $|U_a^*k_b|=|k_{\vp_a(b)}|$ (see Lemma~\ref{identities}) several times. Therefore, $\abs{U_zT_fT_{\bar{g}}k_z(w)}=\abs{f(\vp_z(w))g(z)}$. Using our first assumption we obtain that the first condition in Theorem~\ref{RKT} is satisfied. The second condition is checked similarly.
 \end{proof}
 
 \subsection{RKT for Hankel operators} Next we treat the case of Hankel operators. We first need to define the class of Hankel operators in the setting of strong Bergman-type spaces.  
The Hankel operator $H_u:\Bo\to \Bo^{\perp}$ with symbol $u:\Om\to \C$ is defined by $H_uf=(I-P)uf$, where $P$ is the orthogonal projection of $L^2(\Om;d\sigma)$ onto $\Bo$. It is clear that some additional assumption on $u$ is needed for $H_u$ even to be defined on the whole $\Bo$. Here we provide a condition on $u$ under which $H_u$ is bounded on $\Bo$. Unfortunately, since $H_u$ is not an operator on $\Bo$ we cannot just apply Theorem~\ref{RKT}. Still, basically the same proof gives us the following result. 
 
\begin{cor} Let $\Bo$ be a strong Bergman-type space such that a product of any two reproducing kernels from $\Bo$ is still in $\Bo$.
If $H_{f}$ is a Hankel operator whose symbol $f$ satisfies $$\sup_{z\in\Om}\int_{\Om}\abs{f(\vp_z(w))-f(z)}^pd\sigma(w)<\infty,$$ for some $p>\frac{4-\kappa}{2-\kappa}$ then $H_{f}$ is bounded.
\end{cor} 
 
\begin{proof} The proof is basically the same as for Theorem~\ref{RKT}. Since the linear space of all the normalized reproducing kernels is dense in $\Bo$ it will be enough to show that there exists a constant $C<\infty$ such that $\norm{H_{f}g}\leq C\norm{g}$ for all $g$ that are in the linear span of the reproducing kernels of $\Bo$. Notice first that for any such $g$ we have:
$$H_{f}g(z)= f(z)g(z)-P(fg)(z)=\int_{\Om}(f(z)g(w)-f(w)g(w))\ip{K_w}{K_z}d\sigma(w).$$ Therefore, it is enough to show that the integral operator  
$$Rg(z)=\int_{\Om}\abs{f(z)-f(w)}\abs{\ip{K_w}{K_z}}g(w)d\sigma(w)$$ is bounded on $L^2(\Om;d\sigma)$. For this we again use the Schur's Test (Lemma~\ref{Schur}) with
\[R(z,w) = \abs{ \ip{K_z}{K_w} }\abs{f(z)-f(w)}, \hspace{0.15cm} h(z) \equiv \norm{K_z}^{\al/2}, \hspace{0.15cm}  X = \Om, \hspace{0.15cm} d\mu(z) =  d\nu(z) = d\sigma(z).\]
If $\kappa=0$ set $\al=\frac{4-2\kappa}{4-\kappa}=1$. If $\kappa>0$ choose $\alpha\in (\frac{2}{p}, \frac{4-2\kappa}{4-\kappa})$ such that $q(\alpha-\frac{2}{p})<\kappa$. The condition $p>\frac{4-\kappa}{2-\kappa}$ ensures that such $\alpha$ exists. 
Let $z\in\Om$ be arbitrary and fixed. We have

\begin{eqnarray*} \int_{\Om} R(z, w)\norm{\Kbw}^{\al} \,d\sigma(w) &=& \int_{\Om} \abs{ \ip{\Kbz}{\Kbw}}\abs{f(z)-f(w)}\norm{\Kbw}^{\al} \,d\sigma(w)\\
&=&  \norm{\Kbz}\int_{\Om} \abs{\ip{\kbz}{k_{\vp_z(u)}}}\abs{f(z)-f(\vp_z(u))}\norm{K_{\vp_z(u)}}^{\al-1}\,d\la(u) \\
&\simeq& \norm{\Kbz}\int_{\Om}\abs{\ip{k_{z}}{U^*_zk_{u}}}\abs{f(z)-f(\vp_z(u))}\abs{\ip{\kbz}{k_{u}}}^{1-\al} \,d\la(u) \\
&=& \norm{\Kbz}\int_{\Om}\abs{\ip{U_zk_z}{k_{u}}}\abs{f(z)-f(\vp_z(u))}\abs{\ip{\kbz}{k_{u}}}^{1-\al}\,d\la(u) \\
&=& \norm{\Kbz}^{\al} \int_{\Om}\abs{f(z)-f(\vp_z(u))}\frac{\abs{\ip{\Kbz}{K_{u}}}^{1-\al}}{\norm{K_u}^{2-\al}}\,d\la(u). 
\end{eqnarray*}
Using Holder's inequality we obtain that the last expression is no greater than
$$
\norm{\Kbz}^{\al}\left(\int_{\Om}\abs{f(z)-f(\vp_z(u))}^pd\sigma(u)\right)^{\frac{1}{p}}\left(\int_{\Om}\frac{\abs{\ip{\Kbz}{K_{u}}}^{q(1-\al)}}{\norm{K_u}^{q\left(2-\al-\frac{2}{p}\right)}}\,d\la(u)\right)^{\frac{1}{q}}.
$$
We can use Lemma~\ref{RF} to deduce that the second integral above is bounded by a constant independent of $z$.  Indeed, set $r=q\left(2-\al-\frac{2}{p}\right)$ and $s=r-2q(1-\al)$. By the choice of $p$ and $\al$ we have $r=\frac{p\left(2-\al-\frac{2}{p}\right)}{p-1}=2-\frac{p\al}{p-1}>\kappa$ and $0<s=q\left(\al-\frac{2}{p}\right)<\kappa$ if $\kappa>0$, and $r=s>\kappa$ if $\kappa=0$. Therefore, $r$ and $s$ satisfy all the conditions of \hyperref[A6]{A.6}. Thus
$$
\int_{\Om} R(z, w)\norm{\Kbw}^{\al} \,d\sigma(w)\leq C\left(\sup_{z\in\Om}\int_{\Om}\abs{f(\vp_z(w))-f(z)}^pd\sigma(w)\right)^{1/p} \norm{\Kbz}^{\al},
$$ where $C$ is the constant from \hyperref[A6]{A.6}.  We don't need to check the second condition since the kernel $R(z,w)$ is symmetric. 

\end{proof}

It is interesting to mention that in the classical Bergman space on the disc if $$\sup_{z\in\Om}\int_{\Om}\abs{f(\vp_z(w))-f(z)}^pd\sigma(w)<\infty,$$ holds for some $p\geq 1$ then it holds for all $p\geq 1$. If $p=2$ and $f$ is anti-holomorphic this is equivalent to $\sup_{z\in\Om}\int_{\Om}\abs{H_{f}k_z(w)}^2d\sigma(w)<\infty$. So, the strongest version of the RKT for boundedness holds in this setting. This was first noticed and proved by Axler in~\cite{Ax}. 
 
\section{Reproducing kernel thesis for compactness}\label{RKTComp}

Compact operators on a Hilbert space are exactly the ones which send a weakly convergent sequences into strongly convergent ones. The main goal of this section is to prove that operators that satisfy the conditions from Theorem~\ref{RKT} and send the weakly null $\{k_z\}$ (see Lemma~\ref{compact} below) into strongly null $\{Tk_z\}$ must be compact. More precisely, if an operator $T$ satisfies the conditions from Theorem~\ref{RKT} and $\lim_{d(z,0)\to\infty}\norm{Tk_z}=0$, then $T$ must be compact.

Recall that the essential norm of a bounded linear operator $S$ on $\Bo$ is given by 
$$
\norm{S}_e=\inf\left\{\norm{S-A}: A\textnormal{ is compact on } \Bo\right\}.
$$
We first show two simple results that will be used in the course of the proofs.

\begin{lm}\label{compact} The weak limit of $k_z$ is zero as $d(z,0)\to\infty$. 
\end{lm}

\begin{proof} Note first that property \hyperref[A2]{A.2} implies that if $d(z,0)\to\infty$ then $d(\vp_w(z),0)\to \infty$. Properties \hyperref[A5]{A.5} and \hyperref[A7]{A.7} now immediately imply that $\ip{k_w}{k_z}\to 0$ as $d(z,0)\to\infty$. The fact that the normalized reproducing kernels are dense in $\Bo$ then implies $k_z$ converges weakly to $0$ as $d(z,0)\to\infty$. 
\end{proof}

\begin{lm} 
\label{lm:Compactsaregood}
For any compact operator $A$ and any $f\in\Bo$ we have that $\norm{A^zf}\to 0$ as $d(z,0)\to\infty$.
\end{lm}

\begin{proof} If $f=k_w$ then using the previous lemma we obtain that $\norm{A^zk_w}\simeq \norm{U_zAk_{\vp_z(w)}}\to 0$ as $d(z,0)\to\infty$. For the general case, choose $f\in\Bo$ arbitrary of norm $1$.  We can approximate $f$ by linear combinations of normalized reproducing kernels and in a standard way we can deduce the same result. 
\end{proof}

\begin{thm}\label{RKTC} Let $T:\Bo\to\Bo$ be a linear operator and $\kappa$ be the constant from \hyperref[A6]{A.6}. If   
\begin{equation}\label{e1} \sup_{z\in\Om}\norm{U_zTk_z}_{L^p(\Om;d\sigma)}<\infty \hspace{.3cm} \text{and} \hspace{.3cm}  \sup_{z\in\Om}\norm{U_zT^*k_z}_{L^p(\Om;d\sigma)}<\infty,
\end{equation} for some $p>\frac{4-\kappa}{2-\kappa}$, then
\begin{itemize}
\item[(a)] $ \|T\|_e\simeq \sup_{\norm{f} \leq 1}\limsup_{d(z,0)\to\infty} \norm{T^zf}.$
\item[(b)] If $\lim_{d(z,0)\to\infty}\norm{Tk_z}=0$ then $T$ must be compact.
\end{itemize}
\end{thm}

In the case of the classical Bergman space, part (b) of our result corresponds to the main result in \cite{MZ}*{Theorem 1.2}. In the case of the unweighted Bargmann-Fock space, statement (b) corresponds to \cite{CWZ}*{Theorem 5}.

The following localization property will be a crucial step towards estimating the essential norm. A version of this result in the classical Bergman space setting was first proved by Su\'arez in~\cite{Sua}. Related results were later given in~\cites{MW,MSW,BI}.

\begin{prop}\label{MainEst1} Let $T:\Bo\to\Bo$ be a linear operator and $\kappa$ be the constant from \hyperref[A6]{A.6}. If  
\begin{equation} \sup_{z\in\Om}\norm{U_zTk_z}_{L^p(\Om;d\sigma)}<\infty \hspace{.3cm} \text{and} \hspace{.3cm}  \sup_{z\in\Om}\norm{U_zT^*k_z}_{L^p(\Om;d\sigma)}<\infty
\end{equation} for some $p>\frac{4-\kappa}{2-\kappa}$, then for every $\epsilon > 0$ there exists $r>0$ such that for the covering $\FF_r=\{F_j\}$ (associated to $r$) from Proposition~\ref{Covering}  
\begin{eqnarray*} \norm{ T-\sum_{j}M_{1_{F_j} }TPM_{1_{G_j} }}_{\BL} < \epsilon. \end{eqnarray*}
\end{prop}

\begin{proof}  Again we will use Schur's Test (Lemma~\ref{Schur}) with
\[R(z,w) = \sum_j 1_{F_j}(z)1_{G^c_j}(w)\abs{ \ip{T^*\Kbz}{\Kbw} }, \hspace{0.15cm} h(z) \equiv \norm{\Kbz}^{\al/2}, \hspace{0.15cm}  X = \Om, \hspace{0.15cm} d\mu(z) =  d\nu(z) = d\sigma(z).\] If $\kappa=0$ then set $\alpha=\frac{4-2\kappa}{4-\kappa}=1$. If $\kappa>0$ first choose $p_0$ such that $\frac{2}{p}<\frac{2}{p_0}<\frac{4-2\kappa}{4-\kappa}$ and denote by $q_0$ the conjugate of $p_0$, $q_0=\frac{p_0}{p_0-1}$. Then choose $\alpha\in (\frac{2}{p_0}, \frac{4-2\kappa}{4-\kappa})$ such that $q_0(\alpha-\frac{2}{p_0})<\kappa$. The condition $p>\frac{4-\kappa}{2-\kappa}$ ensures that such $p_0$ and $\alpha$ exist. 
Let $z\in\Om$ be arbitrary and fixed. Since $\{F_j\}$ forms a covering for $\Om$ there exists a unique $j$ such that $z\in F_j$. Now, we have

\begin{eqnarray*} \int_{\Om} R(z, w)\norm{\Kbw}^{\al} \,d\sigma(w) &=& \int_{\Om}  1_{F_j}(z)1_{G^c_j}(w)\abs{ \ip{T^*\Kbz}{\Kbw}}\norm{\Kbw}^{\al} \,d\sigma(w)\\
&=& \int_{G_j^c}  1_{F_j}(z)\abs{ \ip{T^*\Kbz}{\Kbw}}\norm{\Kbw}^{\al} \,d\sigma(w)\\
&\leq& \int_{D(z, r)^c} \abs{ \ip{T^*\Kbz}{\Kbw}}\norm{\Kbw}^{\al} \,d\sigma(w)\\
&=&  \norm{\Kbz}\int_{D(0,r)^c} \abs{\ip{T^*\kbz}{k_{\vf_z(u)}}}\norm{K_{\vf_z(u)}}^{\al-1}\,d\la(u) \\
&\simeq& \norm{\Kbz}\int_{D(0,r)^c}\abs{\ip{T^*U^*_zk_{0}}{U^*_zk_{u}}}\abs{\ip{\kbz}{k_{u}}}^{1-\al} \,d\la(u) \\
&=& \norm{\Kbz}\int_{D(0,r)^c}\abs{\ip{T^{*z}k_{0}}{k_{u}}}\abs{\ip{\kbz}{k_{u}}}^{1-\al}\,d\la(u) \\
&=& \norm{\Kbz}^{\al} \int_{D(0,r)^c}\frac{\abs{T^{*z}k_{0}(u)}\abs{\ip{\Kbz}{K_{u}}}^{1-\al}}{\norm{K_u}^{2-\al}}\,d\la(u). 
\end{eqnarray*}
Using Holder's inequality we obtain that the last expression is no greater than
$$
\norm{\Kbz}^{\al}\left(\int_{D(0,r)^c}\abs{T^{*z}k_0(u)}^{p_0}\,d\sigma(u)\right)^{\frac{1}{p_0}}\left(\int_{\Om}\frac{\abs{\ip{\Kbz}{K_{u}}}^{q_0(1-\al)}}{\norm{K_u}^{q_0\left(2-\al-\frac{2}{p_0}\right)}}\,d\la(u)\right)^{\frac{1}{q_0}}.
$$
We can use Lemma~\ref{RF} to deduce that the second integral above is bounded by a constant independent of $z$. Indeed, set $r=q_0\left(2-\al-\frac{2}{p_0}\right)$ and $s=r-2q_0(1-\al)$. By the choice of $p_0$ and $\al$ we have $r=\frac{p_0\left(2-\al-\frac{2}{p_0}\right)}{p_0-1}=2-\frac{p_0\al}{p_0-1}>\kappa$, and $0<s=q_0\left(\al-\frac{2}{p_0}\right)<\kappa$. if $\kappa=0$ then $r=s=q_0\left(\al-\frac{2}{p_0}\right)>0$. Therefore, $r$ and $s$ satisfy all the conditions of Lemma~\ref{RF}.  

Using that (by Lemma~\ref{identities}) $$\abs{T^{*z}k_0}\simeq\abs{U_zT^*k_z} \in L^p(\Om;d\sigma),$$ and that $\sigma$ is a finite measure, one more application of Holder's inequality on the first integral above gives that this integral is also bounded by a constant (independent of $z$) that goes to $0$ as $r$ approaches $\infty$.
Thus, the first condition of Lemma~\ref{Schur} is satisfied with a constant $o(1)$ as $r\to\infty$. 

Next, we check the second condition. Fix $w\in\Om$. Let $J$ be a subset of all indices  $j$ such that $w\notin G_j$. Then $\cup_{j\in J}F_j\subset D(w,r)$ and consequently 
\begin{eqnarray*}
 \int_{\Om} R(z, w)\norm{\Kbz}^{\al}\,d\sigma(z) &=& \int_{\cup_{j\in J}F_j} \abs{ \ip{T\Kbw}{\Kbz} }\norm{\Kbz}^{\al}\,d\sigma(z)\\ &\leq&  \int_{D(w,r)^c} \abs{ \ip{T\Kbw}{\Kbz} }\norm{\Kbz}^{\al}d\sigma(z). 
 \end{eqnarray*}
Using exactly the same estimates as above, but interchanging the roles of $T$ and $T^*$ we obtain that the last expression is bounded by  $C(r)\norm{\Kbw}^{\al}$, with $C(r)=o(1)$ as $r\to \infty$. 

Let $f\in \Bbt$ with  $\norm{f} \leq 1$. Applying Schur's Test we obtain \begin{eqnarray*} \norm{ \sum_{j}M_{1_{F_j} }TPM_{1_{G_j^c} }f}^2 &=& \norm{\int_{\Om}R(z,w)f(w) d\sigma(w)}^2 \\
&=& o(1) \hspace{0.5cm} \text{ as } \hspace{0.2cm} r \rightarrow \infty.
\end{eqnarray*}
This proves the Proposition for $\kappa>0$. The case $\kappa=0$ can be proved with adaptations of the proof above analogous to the ones in the proof of Theorem~\ref{RKT}.
\end{proof}

\begin{proof}[Proof of Theorem~\ref{RKTC}] 
First, notice that one of these inequalities is fairly easy to deduce. Indeed, using the triangle inequality and the fact that $\lim_{d(z,0)\to\infty}\norm{A^zf}=0$ for every compact operator $A$ (Lemma \ref{lm:Compactsaregood}) we obtain that  
$$
\sup_{\norm{f} \leq 1}\limsup_{d(z,0)\to\infty} \norm{T^zf}\leq \sup_{\norm{f} \leq 1}\limsup_{d(z,0)\to\infty} \norm{(T-A)^zf}\lesssim \norm{T-A}
$$ 
for any compact operator $A$. Now, since $A$ is arbitrary this immediately implies 
$$
\sup_{\norm{f} \leq 1}\limsup_{d(z,0)\to\infty} \norm{T^zf}\lesssim \|T\|_e.
$$

The other side requires more work. The crucial ingredient will be Proposition~\ref{MainEst1}. It is easy to see that the essential norm of $T$ as an operator in $\LL(\Bo)$ is quasi-equal to the essential norm of $T$ as an operator in $\BL$. Therefore, it is enough to estimate the essential norm of $T$ as an operator on $\BL$. Let $\epsilon>0$. By Proposition~\ref{MainEst1} there exists $r>0$ such that for the covering $\FF_r=\{F_j\}$ associated to $r$ 
\begin{eqnarray*} \norm{T- \sum_{j}M_{1_{F_j} }TPM_{1_{G_j} }}_{\BL} < \epsilon. \end{eqnarray*} 

Note that by Lemma~\ref{ToeplitzCompact} the Toeplitz operators $PM_{1_{G_j}}$ are compact. Therefore the finite sum $\sum_{j\leq m}M_{1_{F_j} }TPM_{1_{G_j} }$ is compact for every $m\in \N$. So, it is enough to show that  
$$
\limsup_{m\to\infty}\norm{T_m}_{\BL}\lesssim \sup_{\norm{f} \leq 1}\limsup_{d(z,0)\to \infty} \norm{T^zf}_{\Bbt},
$$
where 
$$
T_m=  \sum_{j\geq m}M_{1_{F_j} }TPM_{1_{G_j} }.
$$
Indeed,
\begin{eqnarray*}
\norm{TP}_e&\leq& \norm{TP- \sum_{j\leq m}M_{1_{F_j} }TPM_{1_{G_j} }}_{\BL}\\
&\leq& \norm{TP- \sum_{j}M_{1_{F_j} }TPM_{1_{G_j} }}_{\BL}+\norm{T_m}\leq \epsilon+\norm{T_m}.
\end{eqnarray*}
Let $f\in\Bo$ be arbitrary of norm no greater than $1$.  Then,

\begin{eqnarray*} 
\norm{T_m f}^2 &=& \sum_{j\geq m}\norm{M_{1_{F_j} }TPM_{1_{G_j} }f}^2\\
&=& \sum_{j\geq m} \frac{\norm{M_{1_{F_j} }TPM_{1_{G_j} }f}^2}{\norm{M_{1_{G_j} }f}^2}\norm{M_{1_{G_j} }f}^2
\leq N\sup_{j\geq m}\norm{M_{1_{F_j} }Tl_j}^2\lesssim\sup_{j\geq m} \norm{Tl_j}^2,
\end{eqnarray*}
where 
$$l_j:=\frac{PM_{1_{G_j} }f}{\norm{M_{1_{G_j} }f}}.$$

Therefore, 
$$\norm{T_m}\lesssim \sup_{j\geq m}\sup_{\norm{f}=1}\left\{\norm{Tl_j}: l_j=\frac{PM_{1_{G_j} }f}{\norm{M_{1_{G_j} }f}}\right\},$$ and hence
$$ \limsup_{m\to\infty}\norm{T_m}\lesssim \limsup_{j\to\infty}\sup_{\norm{f}=1}\left\{\norm{Tg}: g=\frac{PM_{1_{G_j}}f}{\norm{M_{1_{G_j} }f}}\right\}.$$

Let $\epsilon>0$. There exists a normalized sequence $\{f_j\}$ in $\Bo$ such that 

$$ \limsup_{j\to \infty}\sup_{\norm{f}=1}\left\{\norm{Tg}: g=\frac{PM_{1_{G_j}}f}{\norm{M_{1_{G_j} }f}}\right\}-\epsilon\leq \limsup_{j\to\infty}\norm{Tg_j},$$ where
$$g_j:=\frac{PM_{1_{G_j} }f_j}{\norm{M_{G_j}f_j }}=\frac{\int_{G_j}\ip{f_j}{k_w}k_w\,d\la(w)}{\left(\int_{G_j}\abs{\ip{f_j}{k_w}}^2d\la(w)\right)^{\frac{1}{2}}}.$$ Recall that $\abs{U^*_{z}k_w} \simeq\abs{k_{\vp_{z}(w)}}$, and therefore, $U^*_{z}k_w = c(w,z)k_{\vp_{z}(w)}$, where $c(w,z)$ is some function so that $\abs{c(w,z)}\simeq 1$.

For each $j$ pick $z_j\in G_j$. There exists $\rho>0$ such that $G_j\subset D(z_j, \rho)$ for all $j$.  Doing a simple change of variables we obtain 
$$g_j=\int_{\vp_{z_j}(G_j)}a_j(\vp_{z_j}(w))U^*_{z_j}k_w\,d\la(\vp_{z_j}(w)),$$ where  $a_j(w)$ is defined to be 
$$
\frac{\ip{f_j}{k_w}}{c(\vp_{z_j}(w), z_j)\left(\int_{G_j}\abs{\ip{f_j}{k_w}}^2\,d\la(w)\right)^{\frac{1}{2}}}
$$ 
on $G_j$, and zero otherwise. 

We claim that $g_j=U^*_{z_j}h_j$, where  $$h_j(z):=\int_{\vp_{z_j}(G_j)}a_j(\vp_{z_j}(w))k_w(z)\,d\la(\vp_{z_j}(w)).$$ First, using the generalized Minkowski Inequality it is easy to see that $h_j\in L^2(\Om;\sigma)$ and consequently in $\Bo$. To prove that the claim is correct we only need to show that for each $g\in L^2(\Om;\sigma)$ we have that $\ip{g_j}{g}=\ip{h_j}{U_{z_j}g}$. This can be readily done using Fubini's Theorem. 
The total variation of each member of the sequence of measures $\{a_j(\vp_{z_j}(w))\,d\la(\vp_{z_j}(w))\}$, as elements in the dual of $C(\overline{D(0,\rho)})$, satisfies $\norm{a_j(\vp_{z_j}(w))\,d\la(\vp_{z_j}(w))}\lesssim \la(D(0,\rho))$, where the implied constant is only dependent on the one from condition \hyperref[A4]{A.4}. Therefore, there exists a weak-$*$ convergent subsequence which approaches some measure $\nu$. Abusing notation slightly we keep indexing this subsequence by $j$. Let $$ h(z):=\int_{D(0,\rho)}k_w(z)\,d\nu(w).$$ The mentioned weak-$*$ convergence implies that $h_j$ converges to $h$ pointwise. Using the Lebesgue Dominated Convergence Theorem we obtain that $h_j\to h$ in $L^2(\Om;\sigma)$. This implies that $h\in\Bo$. In addition, $1=\norm{g_j}=\norm{U^*_{z_j}h_j}\simeq \norm{h_j}$. Thus, $\norm{h}\lesssim 1$. So, we finally have 
$$ \limsup_{m\to\infty}\norm{T_m}\lesssim \lim_{j\to\infty}\norm{Tg_j}+\epsilon =  \lim_{j\to\infty} \norm{TU^*_{z_j}h_j}+\epsilon\lesssim \limsup_{j\to\infty}\norm{TU^*_{z_j}h}+\epsilon\lesssim  \limsup_{j\to\infty}\norm{T^{z_j}h}+\epsilon $$
and hence $$\limsup_{m\to\infty}\norm{T_m}\lesssim  \sup_{\norm{f} \leq 1}\limsup_{d(z,0)\to \infty} \norm{T^zf}.$$

(b) By a simple density argument it is sufficient to show that $\norm{T^zk_w}\to 0$ as $d(z,0)\to\infty$. But this is clear since $\norm{T^zk_w}\simeq \norm{Tk_{\varphi_z(w)}}$ and $d(\varphi_z(w),0)\simeq d(w,z) \to \infty$ as $d(z,0)\to\infty$. We are done.

\end{proof}

\begin{cor} 
\label{ToeCompact}
Let $\Bo$ be a strong Bergman-type space for which $\kappa>0$. If $T$ is in the Toeplitz algebra $\mathcal{T}_{L^\infty}$ then
\begin{itemize}
\item[(a)] $ \|T\|_e\simeq \sup_{\norm{f} \leq 1}\limsup_{d(z,0)\to\infty} \norm{T^zf}.$
\item[(b)] If $\lim_{d(z,0)\to\infty}\norm{Tk_z}=0$ then $T$ must be compact.
\end{itemize}
\end{cor}
In the case of the unit ball $\B_n$, statement (a) corresponds to \cite{MSW}*{Corollary 5.4}, for the polydisc $\D^n$ \cite{MW}*{Theorem 6.4}. Notice how easy it is to deduce these results when phrased in the more general Bergman-type spaces.

\begin{proof} We first show that if $T$ is a finite sum of finite products of Toeplitz operators with bounded symbols then $T$ satisfies~\eqref{e1} for all $p>1$. By triangle inequality it is enough to consider the case when $T=T_{u_1}T_{u_2}\cdots T_{u_n}$ is a finite product of Toeplitz operators with bounded symbols. Then $$ U_zTk_z=T_{u_1\circ \varphi_{z}}T_{u_2\circ \varphi_{z}}\cdots T_{u_n\circ \varphi_{z}}k_0.$$ It follows now easily from Lemma~\ref{proj} that 
$\sup_{z\in\Om}\norm{U_zTk_z}_{L^p(\Om;d\sigma)}<\infty.$ The proof of the other inequality is basically the same. Therefore, by Theorem~\ref{RKTC}, we have that both (a) and (b) hold if $T$ is a finite sum of finite products of Toeplitz operators with bounded symbols and $\kappa>0$. 

In the general case we prove only part (a) since part (b) is an easy consequence of (a). Let $T$ be an arbitrary operator in the Toeplitz algebra $\mathcal{T}_{L^\infty}$. For any given $\epsilon>0$ there exists $S$ which is a finite sum of finite products of Toeplitz operators such that $\norm{T-S}_e\leq \norm{T-S}<\epsilon$. Clearly, using the supposition that the result is true for finite sums of finite products, $$\norm{T}_e\leq \norm{T-S}_e+\norm{S}_e<\epsilon+ \sup_{\norm{f} \leq 1}\limsup_{d(z,0)\to\infty} \norm{S^zf}.$$ Next, for every $f\in\Bo$ with $\norm{f}\leq 1$ we have $$\norm{S^zf}\leq \norm{(T-S)^zf}+\norm{T^zf}\lesssim \epsilon+\norm{T^zf}.$$ By combining the last two inequalities we easily obtain the desired one for $T$. The opposite inequality holds for general operators $T$ and it was proved in the course of proving the Theorem~\ref{RKTC}.
\end{proof}

In the case of a classical Bargmann-Fock space one can also show (see~\cite{CWZ}*{page 11}) that if $T$ is a finite sum of finite products of Toeplitz operators with bounded symbols then $T$ satisfies~\eqref{e1} for some $p>2$. Therefore the conclusions of the previous Corollary hold for the Bargmann-Fock space as well (even though $\kappa=0$). This gives a much shorter and simpler proof to the corresponding results in~\cite{BI}*{Corollary 6.3}.

Our next goal is to show that the previous corollary holds even with a weaker assumption.  Let $\textnormal{BUC}(\Omega,d)$ denote the algebra of complex-valued functions that are bounded and uniformly continuous on $(\Om,d)$.  Let $\mathcal{T}_\textnormal{BUC}$ denote the Toeplitz algebra generated by the Toeplitz operators with symbols from $\textnormal{BUC}(\Omega,d)$

In this section we show that $\ip{Tk_z}{k_z}\to 0$ as $d(z,0)\to\infty$ already implies that $T$ is compact when $T\in \mathcal{T}_{\textnormal{BUC}}$.  Recall that for a given operator $T$ the Berezin transform of $T$ is a function on $\Om$ defined by
$$
\tilde{T}(z):=\ip{T\kbz}{\kbz}.
$$

\begin{thm}
\label{Berezin} Let $T\in\mathcal{T}_\textnormal{BUC}$. Then  $\limsup_{d(z,0)\to\infty} \tilde{T}(z)=0$ if and only if 
$$
\limsup_{d(z,0)\to\infty} \norm{T^zf}=0
$$ 
for every $f\in\Bbt$. In particular, if the Berezin transform $\tilde{T}(z)$ ``vanishes at the boundary of $\Om$'' then the operator $T$ must be compact.
\end{thm}

For the remainder of this section, SOT will denote the strong operator topology in $\mathcal{L}(\Bo,\Bo)$ and WOT will denote the weak operator topology in $\mathcal{L}(\Bo,\Bo)$.

Key to proving Theorem \ref{Berezin} will be the following lemma.
\begin{lm}\label{SOT} Let $u\in\textnormal{BUC}(\Omega,d)$.  For any sequence $\{z_n\}$ in $\Omega$, the sequence of Toeplitz operators $T_{u\circ \varphi_{z_n}}$ has a \textnormal{SOT} convergent subnet.
\end{lm}

\begin{proof} It is enough to show that the sequence $u\circ \varphi_{z_n}$ has a subnet that converges uniformly on compact subsets of $\Om$. Indeed, assume that this is true. Let $u\circ \varphi_{{z_n}_k}$ be one such subnet which converges to some $v$ uniformly on compact sets. This $v$ will clearly be bounded and continuous on $\Om$. Fix $f\in\Bo$.   Then observe that 
$$\norm{T_{u\circ \varphi_{{z_n}_k}}f-T_vf}^2\leq \norm{(u\circ \varphi_{{z_n}_k}-v)f}^2$$
$$=\int_{\Om\setminus D(0,R)}\abs{(u(\varphi_{{z_n}_k}(w))-v(w))f(w)}^2\,d\sigma(w)+\int_{D(0,R)}\abs{(u(\varphi_{{z_n}_k}(w))-v(w))f(w)}^2\,d\sigma(w).$$
Choosing $R>0$ large enough we can make the first integral arbitrary small. The second integral can then be made arbitrary small by choosing $k$ ``big'' enough. Therefore, $T_{u\circ \varphi_{{z_n}_k}}$ converges in the strong operator topology. 

So, we only need to show that  $u\circ \varphi_{z_n}$ has a subnet that converges uniformly on compact subsets. Property \hyperref[A2]{A.2} and the uniform continuity of $u$  easily imply that $\{u\circ \varphi_{z}: z\in\Om\}$ is equicontinuous. Therefore, it is enough to show that $u\circ \varphi_{z_n}$ has a pointwise convergent subnet. 
Let $\beta\Om$ be the Stone-Cech compactification of $(\Om,d)$. We denote again by $u$ the unique continuous extension of $u$ to $\beta\Om$. By the Tychonoff Compactness Theorem, there exists a subnet $\vp_{{z_n}_k}:\Om\to\beta\Om$ which converges pointwise to some $\vp:\Om\to \beta\Om$. Clearly, then $u\circ \vp_{{z_n}_k}$ is pointwise convergent and so we are done.

\end{proof}

\begin{proof}[Proof of Theorem~\ref{Berezin}]   Suppose that  $\limsup_{d(z,0)\to\infty} \norm{T^zf}=0$ for every $f\in\Bbt$. Taking $f\equiv k_0$ we easily obtain $\abs{\tilde{T}(z)}\leq \norm{Tk_z}$, which implies that $\limsup_{d(z,0)\to\infty} \tilde{T}(z)=0$. 

In the other direction, suppose that  $\lim_{d(z,0)\to\infty} \abs{\tilde{T}(z)}=0$ but  $\limsup_{d(z,0)\to\infty} \norm{T^zf}>0$ for some $f\in\Bbt$. In this case there exists a sequence $\{z_n\}$ with $d(z_n,0)\to\infty$  such that $\norm{T^{z_n}f}\geq c>0.$ We will derive a contradiction by showing that $T^{z_n}$ has a subnet that converges to the zero operator in SOT. 
Notice first that $S_{z_n}$ has a subnet which is convergent in WOT. With slight abuse of notation we continue to denote the subnet by $\{z_n\}$. Denote the WOT limit by $S$. Then $\tilde{T}^{z_n}\to \tilde{S}$ pointwise. The assumption  $\lim_{d(z,0)\to\infty} \abs{\tilde{T}(z)}=0$ implies that $\tilde{T}^{z_n}\to 0$ pointwise as well and hence $\tilde{S}\equiv 0$. Therefore $S$ is the zero operator and consequently $T^{z_n}$ converges to zero in the WOT.

Next, we use the fact that $T$ is in $\mathcal{T}_{\textnormal{BUC}}$ to show that there exists a subnet of $T^{z_n}$ which converges in SOT. Let $\epsilon>0$, then there exists an operator $A$ which is a finite sum of finite products of Toeplitz operators with symbols in $\textnormal{BUC}(\Om,d)$ such that $\norm{T-A}_{\mathcal{L}(\Bo,\Bo)}<\epsilon$.  We first show that $A^{z_n}$ must have a convergent subnet in SOT. By linearity we can consider only the case when $A=T_{u_1}T_{u_2}\cdots T_{u_k}$ is a finite product of Toeplitz operators. As noticed before $$A^{z_n}=T_{u_1\circ \varphi_{z_n}}T_{u_2\circ \varphi_{z_n}}\cdots T_{u_k\circ \varphi_{z_n}}.$$ Now, since a product of SOT convergent nets is SOT convergent, it is enough to treat the case when $A=T_{u}$ is a single Toeplitz operator. But, the single Toeplitz operator case follows directly from Lemma~\ref{SOT}. 

Denote by $B$ the SOT limit of this subnet $A^{{z_n}_k}$. The fact that $T^{z_n}$ converges to zero in the WOT easily implies that $\norm{B}_{\mathcal{L}(\Bo,\Bo)}\lesssim \epsilon$. Now, for every $f\in\Bo$ of norm no greater than $1$ we have $$ \norm{T^{{z_n}_k}f}\leq \norm{A^{{z_n}_k}f}+\norm{(T^{{z_n}_k}-A^{{z_n}_k})f}. $$ Therefore, $\limsup \norm{T^{{z_n}_k}f}\lesssim \norm{Bf}+\epsilon\lesssim 2\epsilon$. Finally, the fact that $\epsilon>0$ was arbitrary implies that $\lim \norm{T^{{z_n}_k}f}=0$ for all $f$. Consequently, we found a subnet $T^{{z_n}_k}$ which converges to the zero operator in SOT. We are done.
 
\end{proof}

\begin{rem}
It is quite possible that the last theorem holds for all operators in the Toeplitz algebra. This is certainly true in the classical Bergman and Bargmann-Fock spaces as a consequence of the fact that each Toeplitz operator $T_u$ can be approximated by Toeplitz operators $T_{B_k(u)}$ with $B_k(u)$ being the so called $k$-Berezin transform of $u$, see any of \cites{BI, MW, MSW, Sua, Sua2} for examples of this phenomenon. It is interesting to see if for Bergman spaces over bounded symmetric domains which satisfy the Rudin-Forelli estimates one can combine the techniques from~\cite{E} and~\cite{Sua2} to prove that
$$
\mathcal{T}_{L^\infty}=\mathcal{T}_{\textnormal{BUC}}.
$$  

It will be also interesting (and probably much harder) to see if this continues to hold in the general setting of Bergman-type spaces.
\end{rem}

\begin{rem} After this work was finished, we found out about the results in~\cite{CWZ}, where simultaneously and independently Theorems~\ref{RKT} and~\ref{RKTC} (b) were proved for the special case of unweighted Fock space.  The authors would like to thank Kehe Zhu, Joshua Isralowitz, and the anonymous referee for their useful comments. 
\end{rem}

\end{document}